\documentclass{amsart}
\pdfoutput=1

\usepackage[utf8]{inputenc}
\usepackage{amsfonts}
\usepackage{import}
\usepackage{mathrsfs}
\usepackage{amsmath}
\usepackage{color}
\usepackage{amsthm}
\usepackage{graphicx}
\usepackage{amscd}

\usepackage[authoryear]{natbib}
\usepackage{lipsum}

\newtheorem{theorem}{Theorem}[section]
\newtheorem{conjecture}{Conjecture}[section]
\newtheorem{definition}[theorem]{Definition}
\newtheorem{lemma}[theorem]{Lemma}

\newtheorem{corollary}[theorem]{Corollary}

\begin{document}

\bibliographystyle{plain}

\markboth{Jes\'us Rodr\'{\i}guez-Viorato}
{Some pretzels}

\title{Alternating Montesinos knots and Conjecture $\mathbb{Z}$}



\author{Jes\'us Rodr\'{\i}guez-Viorato}

\maketitle

\begin{abstract}
Conjecture $\mathbb{Z}$ is a knot theoretical equivalent form of the Kervaire Conjecture. 
We say that a knot have property $\mathbb{Z}$ if it satisfies Conjecture $\mathbb{Z}$ for that specific knot.
In this work, we show that alternating Montesinos knots with three tangles have property $\mathbb{Z}$.
We also show that all the pretzel knots of the form $P(p,q,r)$ (not necessarily alternating) have property $\mathbb{Z}$.
\end{abstract}

\section{Introduction}
The Kervaire Conjecture is a well know combinatorial group theory conjecture. In order to translate Kervaire Conjecture into Knot theory, F. González Acuña and A. Ramírez in \cite{MR2221530} stated a knot theoretical equivalent conjecture called Conjecture $\mathbb{Z}$. 

\begin{conjecture}[Kervaire]\label{conj:kervaire}
For every non-trivial group $G$ the free product group $\mathbb{Z} * G$ can not be normally generated by one element.
\end{conjecture}

\begin{conjecture}[Conjecture $\mathbb{Z}$]\label{conj:z}
If $F$ is a compact orientable and non-separating surface properly embedded in a knot
exterior $E$, then $\pi_1(E/F) \cong \mathbb{Z}$.
\end{conjecture}

In this conjecture, by exterior of a knot, we mean the complement on $S^3$ of an open regular neighborhood of a knot. This implies that $\partial E$ is a torus, and that $F$ can have many boundary components. In fact, when the numbers of components is exactly one,  in \cite{MR2221530} F. González Acuña and A. Ramírez proved that $\pi_1(E/F) \cong \mathbb{Z}$. So, the interesting case would be when $\partial F$ is disconnected. 

We will say that a surface $F$ on the knot exterior $E$ has Property $\mathbb{Z}$ if $\pi_1(E/F) \cong \mathbb{Z}$. Similarly we say that a knot $k$ has property $\mathbb{Z}$ if any $F$ compact orientable and non-separating surface properly embedded in the exterior $E$ of $k$ has property $\mathbb{Z}$. Recently in \cite{JFPretzel} we proved that, by showing property $\mathbb{Z}$ for ICON (\emph{incompressible} compact orientable and non-separating) surfaces, it is enough to show that a given knot $k$ has property $\mathbb{Z}$. Eudave Mu\~noz showed in \cite{eudave-munoz2013} some knots with ICON surfaces with disconnected boundary he proved; that those surfaces have property $\mathbb{Z}$ but it is unknown if the knots have it.

In \cite{MR2221530} it was proved that fibred knots and rank two knots (knots with fundamental group of rank two) have property $\mathbb{Z}$. In \cite{JFPretzel} we were able to prove that many pretzel knots of three braids have Property $\mathbb{Z}$. In the present paper we complete the proof for all pretzel knots with three braids and also that all alternating Montesinos knots with three tangles have property $\mathbb{Z}$ (see Theorem \ref{thm:prop-z-for-alt-mont}).

This result expands the family of known knots that have Property $\mathbb{Z}$. If we were able to expand this family long enough, so that the family is dominant (see Def.  \ref{def:dominant}) the conjecture $\mathbb{Z}$ would be true.

\begin{definition}\label{def:dominant}
 A family of knots $\mathcal{F}$ is \emph{dominant} if for ever knot $k$ in $S^3$ there is a knot $k' \in \mathcal{F}$ such that there is an epimorphism from the fundamental group of $k'$ onto that of $k$.
\end{definition}

Our strategy for tackling Conjecture $\mathbb{Z}$ is by proving that Montesinos knots are a family of dominant knots having Property $\mathbb{Z}$. The present work is a step in that direction.

The main technique is by the classification of incompressible surfaces on Montesinos knots developed by A. Hatcher and U. Oertel in \cite{MR1030987}. We also use the notion of oriented train tracks developed in \cite{JFPretzel}.

\section{Preliminary}

\subsection{About oriented weights}
As mentioned before, in \cite{JFPretzel} we introduce the concept of oriented train tracks, that basically are the same as regular train tracks but instead of using positive integers we use elements of $\mathbb{P}_2$ (the free semigroup of rank two) as weights. In this section, we state some basic properties and lemmas about the weights of oriented train tracks.

From the definition of oriented train tracks (see \cite{JFPretzel}), three functions take an important role: $J, \rho, -: \mathbb{P}_2 \to \mathbb{P}_2$,  defined as:

\begin{eqnarray}
\rho(x_1, x_2, \ldots, x_n) &=& (x_2,x_3,  \ldots, x_n,x_1) \\
J(x_1, x_2, \ldots, x_n) &=& (x_n,x_{n-1},  \ldots, x_{2},x_1)\\
-(x_1, x_2, \ldots, x_n) &=& (-x_1, -x_2, \ldots, -x_n)
\end{eqnarray}

We also denote by $\oplus$ the operation on $\mathbb{P}_2$; this is, $\oplus$ denotes the usual concatenation of words. And we denote by $\rho^n$ as the $n$ times application of $\rho$ when $n >0$ and the $|n|$ times application of $\rho^{-1}$ (the inverse function of $\rho$) when $n$ is negative; $\rho^0$ is the identity function. We use $|x|$ to denote the length of the word $x$. The following properties can be easily verified. 

\begin{lemma}\label{lemma:prop-rho-and-J}Let  $x, y \in \mathbb{P}_2$ and
$\rho$ and $J$ be the functions defined above. Then:
\begin{enumerate}
 \item $\rho^n(\rho^m(x)) =  \rho^{n+m}(x)$ for all $n,m \in \mathbb{Z}$
 \item $\rho(-x) = -\rho(x)$ and $J(-x) = -J(x)$
 \item If $n = |x|$ then $\rho^n(x) = x$.
 \item $J \circ J (x) = x$
 \item If $n = |x|$ then $\rho^n(x \oplus y) = y \oplus x$
 \item $\rho \circ J (x) = J \circ \rho^{-1} (x) $
\end{enumerate}

\end{lemma}

A new function that we introduce in this paper is $\Sigma(x) = x_1 + x_2 +\cdots + x_n $, where $x = (x_1, x_2, \dots, x_n)$ is an element of  $\mathbb{P}_2$; we think of each $x_i$ as an element of $\{+1,-1\}$. An important property of this function is described in the following lemma.

\begin{lemma}\label{lemma:rho-a}
Let $a \in \mathbb{P}_2$ such that $\Sigma(a) = \pm 1$, then $\rho^t(a) = a$ if and only if $t$ is a multiple of $|a|$.
\end{lemma}
\begin{proof}
Let $m = |a|$, we are going to prove that $m | t$. Observe that since $\rho^m(a) = a$ then $\rho^g(a) = a$ for $g =gcd(t,m)$.

Since $g |  m$ it follows that $a = b \oplus b \oplus \dots \oplus b$, $m/g$ times and $|b| = g$. Then, $ \pm 1 = \Sigma(a) = (m/g) \cdot  \Sigma(b) $ therefore $m/g | \pm 1$, so $m=g$.

But, by definition of $g$, we get that $g|t$, so $m|t$. 

The converse is obviously true by Lemma \ref{lemma:prop-rho-and-J}.
\end{proof}

The weights of the train tracks we are considering here satisfy that $\Sigma(a)= \pm 1$; this is mainly because $\partial F$ is $1$ in $H_1(\partial E(k))$.  This is why the previous lemma and  following corollary are going to be frequently used on the study of surfaces throughout this paper.  

\begin{corollary}\label{cor:rho-a=a}
 Let $a \in \mathbb{P}_2$ such that $\Sigma(a) = \pm 1$ and $m=|a|$ is odd. Then $\rho^{2(x-y)}(a) =a $ with $|x|,|y|<m/2$  if and only if $x=y$.
\end{corollary}
\begin{proof}
 By Lemma \ref{lemma:rho-a} we get that $m| 2(x-y)$, since $m$ is odd, we also got that $m|x-y$. But $|x-y| \leq |x| + |y| < m/2 + m/2 = m$, so $x - y = 0$, i.e $x = y$.
\end{proof}

\subsection{Incompressible surfaces of Montesinos knots}
In this subsection, we review the algorithm developed by A. Hatcher and U. Oertel in \cite{MR1030987}. The algorithm allows us to construct all the incompressible surfaces on a Montesinos knot. 

First, recall the definition of diagram $\mathscr{D}$.  The diagram $\mathscr{D}$ is a subset of $\mathbb{R}^2$ with the following simplicial structure (see Fig. \ref{fig:diagrama_D}):
\begin{itemize}
 \item Its vertexes are the points of coordinates $( \frac{q-1}{q}, \frac{p}{q})$ denoted by $\langle p/q \rangle$, 
 \item Also $(1, \frac{p}{q})=\langle p/q \rangle_0$  and $(-1,0)=\langle \infty \rangle = \langle 1/0 \rangle$ are vertexes of $\mathscr{D}$
 \item The edges are segments with ending points at $\langle p/q \rangle$ and $\langle r/s \rangle$ whenever $ps-qr = \pm 1$. We denote these edges by $\langle p/q, r/s\rangle$.  
 \item The segments $\langle \infty , n \rangle$ where $n \in \mathbb{Z}$ are also edges.
 \item The horizontal segments with ending points $\langle p/q \rangle$ and $\langle p/q \rangle_0$ are included in the simplicial structure of $\mathscr{D}$.
 \item Hatcher and Oertel also include 2-simplexes with vertexes on $\langle p/q \rangle$, $\langle r/s \rangle$ and $\langle u/v \rangle$ whenever each pair of vertexes is connected by an edge. But we are not going to make use of it.
\end{itemize}

In order to compute incompressible surfaces in the exterior of a Montesinos knot $M(p_1/q_1, p_2/q_2, \dots. p_n/q_n)$ we need to compute $n$ \emph{edgepaths} $\gamma_{1}$, $\gamma_2$, \dots $\gamma_n$ on diagram $\mathscr{D}$ (see Fig. \ref{fig:diagrama_D}) satisfying the following conditions:
\begin{figure}[th]
 \centering
 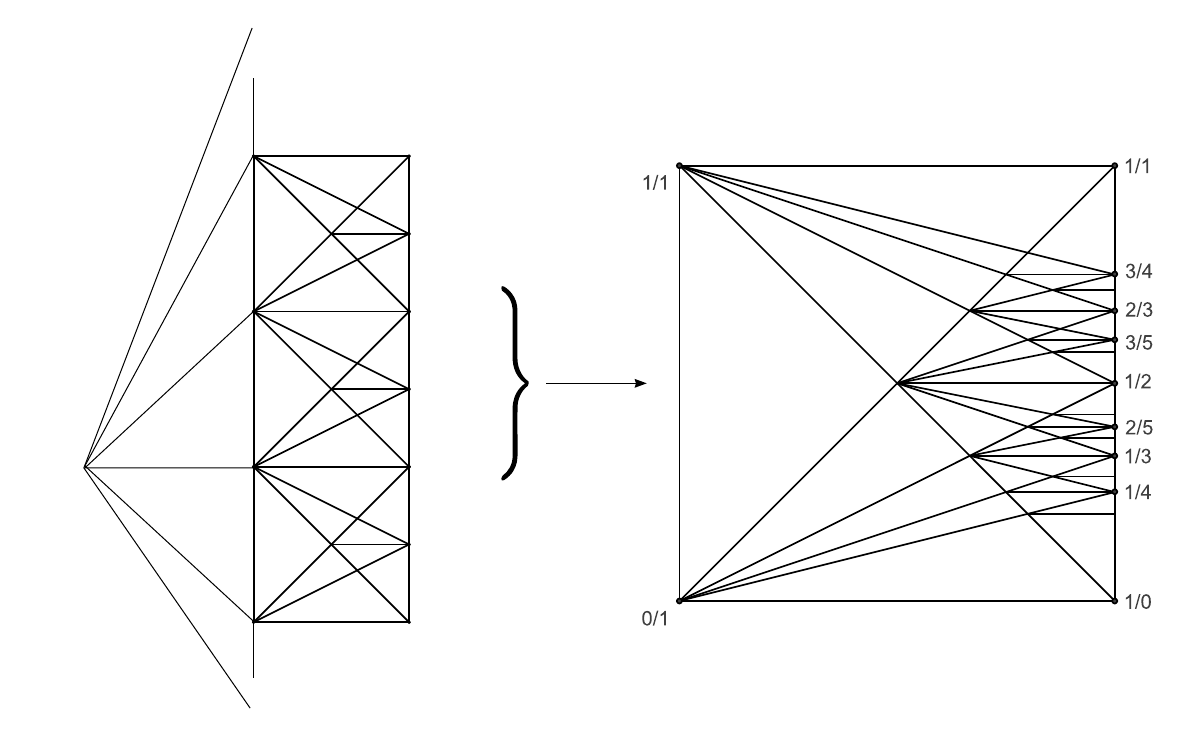
 \caption{Diagram $\mathscr{D}$}
 \label{fig:diagrama_D}
\end{figure}

\begin{itemize}\label{def:condiciones-de-sistema-de-caminos}
\item[(E1)] The starting point of $\gamma_i$ lies on the edge $\langle
p_i/q_i ,  p_i/q_i \rangle$, and if the starting point is not the vertex
$\langle p_i/q_i \rangle$, then the edgepath $\gamma_i$ is constant.

 \item[(E2)] $\gamma_i$ is \emph{minimal}, i.e., it never stops and retraces
itself, nor does it ever go along two sides of the same triangle of
$\mathscr{D}$ in succession.
 
 \item[(E3)] The ending points of the $\gamma_i$'s are rational points of
$\mathscr{D}$ which all lie on one vertical line and whose vertical coordinates
add up to zero. 

 \item[(E4)] $\gamma_i$ proceeds monotonically from right to left,
``monotonically'' in the weak sense that motion along vertical edges is
permitted.
\end{itemize}

Once we have an edgepath system satisfying conditions E1-E4, we can construct surfaces in the exterior of $M(p_1/q_1, p_2/q_2, \dots. p_n/q_n)$ as described in \cite{MR1030987}. Those surfaces are called \emph{candidate surfaces} associated to the edgepath system $\gamma_i$ ($i=1, \dots, n$). We now briefly explain the constructions of candidates surfaces:

\begin{itemize} 
  \item We first start by dividing $S^3$ in $n$ balls $B_1, \dots B_n$, each one containing a rational tangle $p_i/q_i$. 
  
  \item Consider a collar neighborhood of $\partial B_i$ in the interior of $B_i$,  which has product structure ($\partial B_i \times [0,1]$ being $\partial B_i \times \{1\}$ the boundary of $B_i$). Now, the Montesinos knot can be isotoped in such a way that inside that collar neighborhood there is nothing but four transverse arcs ($\{ \text{four points}\}\times [0,1]$) and the rest of the tangle lies over $\partial B_i \times {0}$.
  \item Now, we are going to construct a surface $F_i$ inside $\partial B_i \times [0.1] \subset B_i$ using a Morse theoretical sense, where any level $\partial B_i \times \{t\}$ is an sphere with four holes. The edgepath $\gamma_i$ is going to be used as the list of saddle changes, giving us how the intersection of $F_i \cup (\partial B_i \times \{t\})$ would look like.
  \item First take an integer $m$ such that the product $m|\gamma_i|$ is integer. Here, $|\gamma_i|$ denotes the number of edges traveled by $\gamma_i$, where fractional values are possible as explained in \cite{MR1030987}. The number $m$ is called the \emph{number of sheets}.
  \item Start by adding a set of $m$ parallel surfaces inside of $[\epsilon, 0]$  such that, the intersection of these surfaces with any level $t$ in $[\epsilon, 0]$ is a set of $m$ pairs of parallel arcs of slope $p_i/q_i$.
  \item From now on, we travel the edges of $\gamma_i$ from right to left, and for each edge $\langle a/b, c/d\rangle$ of $\gamma_i$ ($a/b$ goes first on $\gamma_i$) we add $m$ saddles, each saddle takes a pair of arcs of slope $a/b$ and transform them into a pair of arcs of slope $c/d$ (there are two possible choices for each saddle), after traveling the edge $\langle a/b, c/d\rangle$ we end with a surface that at his upper level is a set of $m$ parallel pairs arcs of slope $c/d$.
  \item The last edge of $\gamma_i$ can be a fraction of an edge, so it starts at vertex $\langle a/b \rangle$ and ends at a point of the form $(m-\alpha)/m\langle a/b \rangle + \alpha/m\langle c/d \rangle$, meaning that we going to use $\alpha$ saddles instead of $m$.
  \item This gives us a set of surfaces $F_i$ on each ball $B_i$, we can glue them together to form a \emph{candidate surface} for $k$. This gluing is possible thanks to condition (E3) for the edgepath system.
  \item For simplicity we omit many details of this construction, in fact we omit the construction of the constant edgepath and the construction of extended candidate surface. We are not going to make much use of these cases. But for a complete detailed construction we refer to \cite{MR1030987}.
\end{itemize}

It is possible to compute the boundary slope of a candidate surface through their edgepath system. For that, we define the twisting of $\gamma_i$ as $$\tau(\gamma_i) = 2(e_--e_+)$$ where $e_+$($e_-$) is the number of edges on $\gamma_i$ that increase (decrease) the slope,
and fractional values of $e_\pm$ corresponding to the final edge are allowed.
The part of $\gamma_i$ on an edge with a vertex on $\langle \infty \rangle$ do not contribute to the value of $\tau(\gamma_i)$.

Now, if $F$ is a candidate surface associated to the edgepath system $\gamma_i$ satisfying conditions E1-E4, we define the twist of $F$ as $$\tau(F) = \tau(\gamma_1) + \tau(\gamma_2)+ \cdots + \tau(\gamma_n).$$ Now, the boundary slope of $F$ can be easily computed as 

\begin{equation}\label{eqn:slope}
m(F) = \tau(F) - \tau(F_0) 
\end{equation}

where $F_0$ is a Seifert surface for $M(p_1/q_1, p_2/q_2, \dots. p_n/q_n)$. The Seifert surface can also be described as an edgepath system, an explanation of this process can be found in \cite{MR1030987}. 

One thing to notice is that if $p_i =1$ and $q_i > 1$ the edgepath $\gamma_i$ has two possible ways to travel from right to left until axis $x=0$; one is non-decreasing $\gamma^+_i$ and the other one is non-increasing $\gamma_i^-$. We can write the precise formulas as:

\begin{eqnarray}
\gamma^+_i(x) & = & \left\{ \begin{array}{ll}
                     1 -x &\ \text{for}\ x\in[0,1-1/q_i]\\
                     1/q_i &\ \text{for}\ x > 1-1/q_i
                    \end{array}
                    \right.\nonumber\\
\gamma^-_i(x) & = & \left\{ \begin{array}{ll}
                             x/(q_i-1)&\ \text{for}\ x\in[0,1-1/i]\\
                             1/q_i&\ \text{for}\ x > 1-1/i
                            \end{array}
                            \right. \label{eqn:gama_i}
\end{eqnarray}

When $p_i=-1$ there are also two formulas, they are exactly as above but multiplied by $-1$:
\begin{eqnarray}
\gamma^+_{-1/q_i}(x) & = & -\gamma_{1/q_i}^-(x)\nonumber \\
\gamma^-_{-1/q_i}(x) & = &  -\gamma_{1/q_i}^+(x) \label{eqn:gama_-i}
\end{eqnarray}

Now, when the endings of the edgepath system have positive $x$-coordinate, we say that the associated candidate surfaces are of type I. So, to compute the type I candidate surfaces for a given pretzel  $P(p,q, r) = M(1/p,1/q,1/r)$ we must solve eight ``linear'' equations resulting from the following identity: 
\begin{equation}\label{eqn:condit-E3-eqn}
\gamma_p^{\pm}(x)+\gamma_q^{\pm}(x)+\gamma_r^{\pm}(x) = 0 
\end{equation}

Once we find a solution for the equation \eqref{eqn:condit-E3-eqn}, we can compute the twist. When the tangle has the form $1/q$, The following formula gives us the twist in terms of $x$ and $q$:
 
 \begin{equation}\label{eqn:tao_gamma_i}
 \begin{array}{r@{}l}
  \tau(\gamma^+) &{}= \frac{2}{1-x} -2q \\
  \tau(\gamma^-) &{}= 2 - \frac{2x}{(1-x)(q-1)}
 \end{array}
 \end{equation}where $x$ satisfies that $ 0 \leq x \leq 1-1/q$. If we require to compute this formula for negative tangles $-1/q$, we just multiply by $-1$:

 \begin{eqnarray}
\tau(\gamma^+_{-1/q_i}) & = & -\tau(\gamma_{1/q_i}^-)\nonumber \\
\tau(\gamma^-_{-1/q_i}) & = &  -\tau(\gamma_{1/q_i}^+) \label{eqn:tao_gama_-i}
\end{eqnarray}
 
\subsection{The minimum number of sheets}

To construct a candidate surface given an edgepath system $\gamma_i$, the first step is to choose a number $m$ called \emph{number of sheets} (see \cite{MR1030987}). This number can be chosen from any number satisfying that $m |\gamma_i|$ is an integer for all $i$. From those possible values of $m$, the smallest divides all other possible values; we called it \emph{minimum number of sheets}. The algorithm programed by Nathan M. Dunfield in \cite{NathanM2001309} to compute the boundary slopes of a Montesinos knots, only require to use the minimum number of sheets. But for our problem, we have to consider all possible values for the number of sheets. 

\begin{definition}
 Given a system $\gamma_i$ of edgepaths, we define the minimum sheets of $\gamma_i$ as the smallest integer $m$ such that $m|\gamma_i|$ is integer for every $i$.
\end{definition}

In the next lemma, we will use the concept of monochromatic edgepath. This concept require coloring the edges of $\mathcal{D}$, this is done as follows. There are only three types of vertexes $\langle 0/1 \rangle$, $\langle 1/0 \rangle$ and $\langle 1/1 \rangle$ modulo two reduction. Observe that all non-horizontal edges have two different types of vertexes. So we can say that the color of an edge is determined by the type of its ends, meaning that there are three colors for the edges, $\langle 0/1 , 1/0 \rangle$, $\langle 1/0 , 1/1 \rangle$ and $\langle 1/1 ,  0/1 \rangle$. 

\begin{definition}
 If an edgepath $\gamma$ does not go along two edges of different color we say that it is \emph{monochromatic}.
\end{definition}

Observe that, by definition, all edgepaths with length less than or equal to one are monochromatic (including constant edgepaths).

\begin{lemma}\label{thm:monochromatic-system-use-minimum-sheets}
 Let $\gamma_i$ be a system of monochromatic edgepaths satisfying conditions E1-E4. If one of the associated candidate surfaces is orientable and connected, then it has the minimum or twice the minimum number of sheets.
\end{lemma}

\begin{proof}
  By \cite[Theorem 3.6 and Theorem 3.5]{JFPretzel} the diagram of possible orientations on a monochromatic edgepath does not depend on the choice of saddles, this implies that the number of components and orientability of a candidate surface does not depend on the choice of saddles. Then, if we fix the number of sheets $m$, all the candidate surfaces associated to the edgepath system with $m$ sheets have the same number of components and the same type of orientability. 
  
  Let $m'$ be the minimum number of sheets, observe that $m'$ divides $m$. By the previous paragraph, we only need to construct a non-connected or non-orientable candidate surface for any multiple $m$ of $m'$ greater than $2m'$.
  
  Let $F'$ be the candidate surface associated to the given edgepath system but with $m'$ sheets. If $F'$ is orientable, by taking $m/m'$ parallel copies of $F'$, we will obtain an orientable candidate surface $F$ associated to the edgepath system with $m$ sheets, but with $m/m' > 2$ connected components. And we are done in this case.
  
  In case that $F'$ is non-orientable, we can take $F'' = \partial \eta (F')$ the boundary of a regular neighborhood of $F'$. This surface is now orientable and has $2m'$ sheets. Applying the same constructions as above, we can find a candidate surface with $m = 2km'$ sheets, orientable, but with $k > 1$ connected components. But if $m$ is an odd multiple of $m'$ (let say, $m = (2t+1)m'$ ) we can construct a candidate surface of $m$ sheets by taking $t$ parallel copies of $F''$ and one copy of $F'$, obtaining in this way a non-orientable candidate surface.
\end{proof}

The previous lemma allows us, in the case of monochromatic edgepaths, to restrict our attention to the minimum number of sheets; twice is not relevant for us, because for ICON surfaces the number of sheets has to be odd.

\begin{corollary}\label{cor:vertexes-implies-coonected-bounday}
 Let $\gamma_i$ be a system of edgepaths satisfying conditions E1-E4 with ends at vertexes of $\mathcal{D}$, then the associated candidate surface will be connected, orientable and non-separating if and only if its boundary is connected. 
\end{corollary}
\begin{proof}
  Observe that the minimum number of sheets is one, because the edgepath system has ends at vertexes (the lengths $|\gamma_i|$ are integers).
  
  By \cite[Corollary 3.8]{JFPretzel}, non-monochromatic and ending at vertexes of $\mathcal{D}$ will imply non-orientable. Then, all the edgepaths are monochromatic, and by the previous theorem (Theorem \ref{thm:monochromatic-system-use-minimum-sheets}) the number of sheets has to be $1$ or $2$. But non-separating implies that the number of sheets is odd, so the number of sheets is equal to one. Then the boundary is connected.
\end{proof}

\section{Alternating Montesinos knots}

 Recall that all candidate surfaces can be divided into three types: I, II and III. First we start by proving that no alternating Montesinos knots has a type II or III ICON candidate surface with disconnected boundary. There is another type of candidate surfaces, named \emph{extended candidate surfaces} in \cite{MR1030987}, but we are not considering them here, because as noted in \cite{MR1030987}, there are no extended candidate surfaces for three tangles Montesinos knots.
 
\begin{theorem}\label{thm:no-typeii}
 Let $S$ be an orientable type II candidate surface with $m$ sheets, where $m$ is an odd number. Then $S$ is connected if and only if $m=1$.
\end{theorem}

\begin{proof}
As we seen in \cite[Theorem 3.9]{JFPretzel}, if a type II candidate surface is associated to an edgepath system consisting of monochromatic edgepaths of color $\langle 1/1, 0/1\rangle$ then it has exactly $m$ connected components. 

Now, if there are other colors, they can not be extended more than a fraction of an edge (by monochromaticity condition in \cite[Corollary 3.8]{JFPretzel}) so they are monochromatic or quasi-monochromatic. Then, the diagram of possible orientation can be described using \cite[Theorem 3.7]{JFPretzel}. 

As vertical edgepaths are of color $\langle 1/1, 0/1\rangle$, the possible edgepaths for type II candidate surfaces can be one of the following: 

\begin{itemize}
 \item Red: Monochromatic of color $\langle 1/0 , 0/1 \rangle$ with possibly a fraction of an edge of color  $\langle 1/1 , 0/1 \rangle$ (quasi-monochromatic)
 \item Green: Monochromatic of color $\langle 1/0 , 1/1 \rangle$ with possibly a fraction of an edge of color $\langle 1/1, 0/1 \rangle$ (quasi-monochromatic)
 \item White: Monochromatic of color $\langle 1/1 ,  0/1 \rangle$.
\end{itemize}

The cases Red and Green can be quasi-monochromatic or monochromatic depending on whether or not the edgepath contains a fraction of a vertical edge. On both cases, the diagram of possible orientations described in \cite[Theorem 3.7]{JFPretzel} applies here if we allow $r$ and $s$ to be zero. The diagram of possible orientations is drawn in Fig. \ref{fig:diag-wrg-cases}.

\begin{figure}\label{fig:diag-wrg-cases}
  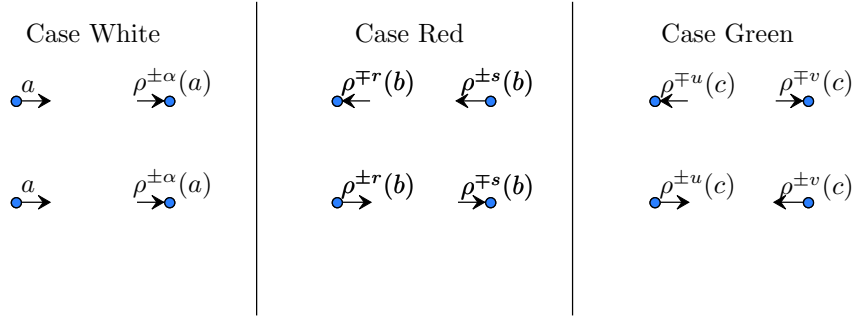
  \caption{Diagram of orientations for the White, Red and Green cases}
\end{figure}

The first case is when all edgepaths are of the same color, without lost of generality, let say they are Red.  So when we identify two diagrams we obtain the relations: $\rho^{\pm s_1}(b_1) = \rho^{\mp r_2}(b_2)$ and $\rho^{\mp s_1}(b_1) = \rho^{\pm r_2}(b_2)$, ($r_2$ and $s_1$ can be zero) from these equations we deduce that $\rho^{2(\pm s_1 \pm r_2)}(b_1) = b_1$. By Corollary \ref{cor:rho-a=a} it follows that $\pm s_1 = \mp r_2$. Then, as $\rho^{\pm s_1}(b_1) = \rho^{\mp r_2}(b_2)$ we can cancel $\rho^{\pm s_1}$ and obtain $b_1 =b_2$. In resume, the relations obtained after identifying two tangles, are equivalent to $b_i = b_{i+1}$. So for every value of $b_1$ we get a solution of the equation system, so the candidate surface will have as many connected components as values has $b_1$. But, recall that the first $\max\{r_1,s_1\}$ letters of $b_1$ are equal to their corresponding last ones, i.e., $b_1 = c \oplus d \oplus -J(c)$ where $|c|=\max\{r_1, s_1\}$ and $|d|$ is odd. This implies that $F$ has $|d| + |c|$ connected components, but because $|d|$ is odd,  $F$ will be connected only if $m=1$.

Analogously, we can apply this argument if all edgepaths were Green and, as we saw in \cite[Theorem 3.9]{JFPretzel}, if all edgepaths were White. 

Now, the remaining cases are when there are at least two edgepaths of different color. In that case, we have to analyze three different equations resulting after gluing.

\emph{Case White with Red}. As we can see on Fig. \ref{fig:diag-wrg-cases}, after identifying the right side of the White diagram of orientations with the left side of the Red diagram, we obtain the following system of equations:

\begin{eqnarray*}
\rho^{\pm\alpha}(a)&=&-J(\rho^{\mp r}(b))\\
\rho^{\pm\alpha}(a)&=& \rho^{\pm r}(b)
\end{eqnarray*}

Combining this system of two equations, we get that $\rho^{\pm r}(b) = \rho^{\pm\alpha}(a) = -J(\rho^{\mp r}(b)) = \rho^{\pm r} (-J(b))$. Then, after canceling $\rho^{\pm r}$ we get that $b = -J(b)$, but this will imply that the value in the middle of $b$ will be equal to its negative. The same will happen if we identify the left side of the White diagram with the right side of the Red one. This means that in this case candidates surfaces with an odd number of sheets are non-orientable.

\emph{Case White with Green}. This is analogous to the previous one, it is impossible because after identifying White with Green, we will get similar equations as before.

\emph{Case Red with Green}. After identifying the right side of the Red diagram with left side of the Green we obtain the following equations (see Fig. \ref{fig:diag-wrg-cases}):

\begin{eqnarray}\label{eqn:red-green-eqns}
\rho^{\pm s}(b)&=&-J(\rho^{\mp r}(c))\\
\rho^{\mp s}(b)&=& \rho^{\pm r}(c)
\end{eqnarray}

This will get us $\rho^{2(\pm s \pm r )}(c) = c$, but by Corollary \ref{cor:rho-a=a}, this is only possible if $\pm s = \mp r$, so $\rho^{\mp s} = \rho^{\pm r}$ . Now canceling $\rho^{\mp s} = \rho^{\pm r}$ on the second equation of \eqref{eqn:red-green-eqns} we obtain $b=c$. After replacing $\mp r$ by $\pm s$ and $c$ by $b$ on the first equation \eqref{eqn:red-green-eqns} we get that $\rho^{\pm 2s}(b) = -J(b)$, then, taking the sum of both sides we get that $\Sigma(b) = \Sigma(\rho^{\pm 2s}(b)) = \Sigma(-J(b)) = -\Sigma(b)$, this implies that $\Sigma(b) = 0$ which is a contradiction. So we can not identify Red and Green diagrams. 


\end{proof}

\begin{theorem}\label{thm:no-typeiii}
  Le $S$ be an orientable type III candidate surface with an odd number of sheets $m$. Then, $S$ is connected if and only if $m=1$.
\end{theorem}
\begin{proof}

Recall that an edgepath that arrives at $x=0$, has to be monochromatic at that point; otherwise the associated candidate surface would be non-orientable. But when the edgepath keeps moving to $\langle \infty \rangle$, it can keeps being monochromatic or change to quasi-monochromatic. It is not hard to see, by analyzing the diagram $\mathcal{D}$, that there are only four possible types of edgepaths according to the edge colors they pass through. We will refer to these cases as follows:

\begin{itemize}
 \item \emph{Monochromatic Red}. The edgepath travels only along edges of type $\langle 1/0, 0/1\rangle$
 \item \emph{Monochromatic Green}. The edgepath travels only along edges of type $\langle 1/0, 1/1\rangle$
 \item \emph{Quasi-monochromatic White-Red}. The edgepath travels only along edges of color $\langle 1/1, 0/1\rangle$, but in the last edge it travels along one edge of color $\langle 0/1, 1/0 \rangle$
 \item \emph{Quasi-monochromatic White-Green}. The edgepath travels along edges of color $\langle 1/1, 0/1\rangle$ but in the last edge it travels along one edge of color $\langle 1/1, 1/0 \rangle$
\end{itemize}

For each of these types of edgepaths we can depict the diagram of possible orientations using \cite[Theorem 3.6]{JFPretzel} at $x=0$. To obtain the final diagram of orientations (at $x>0$) we have to analyze what happens with the diagram of orientations when passing trough a saddle. Recall that the saddle has to create a pair of arcs of infinity slope.  The two possible saddles (under ambient isotopy) are depicted on Fig. \ref{fig:saddle-crating-infty-slope}; notice that we marked one saddle with $r$ and the other with $s$; this means, that we are taking $r$ and $s$ parallel saddles, respectively. Moreover, for each edgepath $\gamma_i$, we are going to denote by $r_i$ the number of saddles of the corresponding type and with $s_i$ the other. 

By condition (E3), the value $\alpha  = r_i+s_i$ is the same for all $i$. And as we are not beyond $\infty$ and we are not at type II case, it follows that $0<\alpha\leq m$.

\begin{figure}\label{fig:saddle-crating-infty-slope}
  \centering
  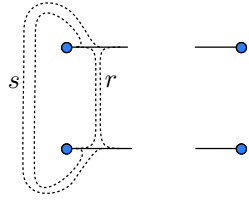
  \caption{Saddles that generate infinity slope pair of arcs.}
\end{figure}

Let us return to the idea raised above. For each of the type of edgepaths above (Monochromatic Red, Green, etc.), we compute the diagram of possible orientation at $x=0$, and then apply their corresponding $\alpha$ saddles. The resulting diagrams are the ones in Fig. \ref{fig:typeIII-mono-cases} and in Fig. \ref{fig:typeIII-quasimono-cases}. For the monochromatic case, do not forget that the first $\max\{r_i,s_i\}$ letters of $a_i$ are equal to the the corresponding last ones but with opposite signs.

\begin{figure}
\centering
   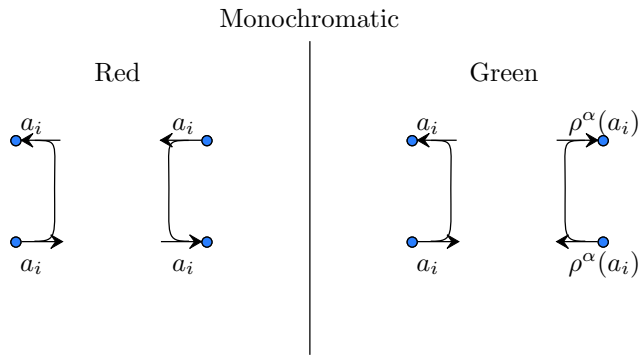
  \caption{Diagram of possible orientations for type III monochromatic cases}
  \label{fig:typeIII-mono-cases}
\end{figure}

\begin{figure}
\centering
   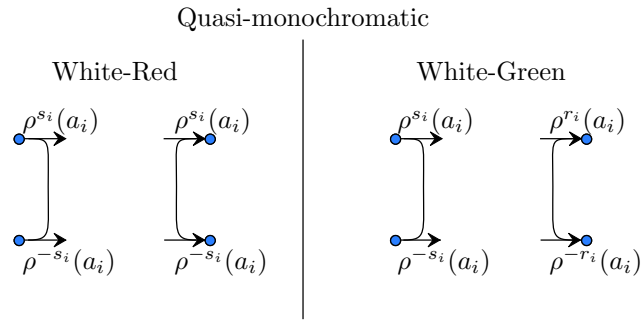
  \caption{Diagram of possible orientations for type III quasi-monochromatic cases}
  \label{fig:typeIII-quasimono-cases}
\end{figure}

When we identify the right side of a monochromatic edgepath (green or red) with the left side of a quasi-monochromatic edgepath or vice versa we obtain a pair of relations as follows:
\begin{eqnarray}
 -J(a_i) &=& \rho^{s_i}(a_{i+1}) \\
 a_i &=& \rho^{-s_i}(a_{i+1})
\end{eqnarray}

After some simple manipulations (see properties of $J$ and $\rho$ at Lemma \ref{lemma:prop-rho-and-J}), we get that $a_{i+1} = -J(a_{i+1})$. This contradicts that $\Sigma(a_{i+1}) = \pm 1$. So, the candidate surface would be non-orientable.  Then, the edgepaths $\gamma_i$ are all monochromatic or all quasi-monochromatic. 

\emph{Case on which all edgepath are monochromatic}. If we glue the right side of a red diagram of orientations with any other monochromatic (red or green), the relations obtained can be reduced to $a_i=a_{i+1}$. But if instead we glue the right side of a green diagram of orientations with any other diagram, the relations can be reduced to $a_i = -\rho^{\alpha}\circ J(a_{i+1})$.

Observe that the function $G = -\rho^{\alpha} \circ J$ has order two, that is, $G \circ G (a) = a$ for all $a \in \mathbb{P}_2$. Then, after considering all the relations resulting from the gluing, we end up with the $a_i's$ equal to $a_1$ or to $G(a_1)$.  One thing to notice is that $a_1 = G(a_1)$ can not happen, because $\Sigma(a_1) = \Sigma(G(a_1)) = \Sigma( -\rho^{\alpha} \circ J(a_1)) = -\Sigma(a_1)$, that implies that $\Sigma(a_1) = 0$, contradicting the fact that $\Sigma(a_1) = \pm 1$. 

Then, for any value of $a_1$, we get a solution for the set of relations. We conclude that the associated candidate surface have $|a_1|=m$ connected components. 

\emph{Case on which all edgepath are quasi-monochromatic}. This case is similar to the previous one. The main difference is the way in which the relations are manipulated. For instance, if we identify the right side of a diagram on Fig. \ref{fig:typeIII-quasimono-cases} and the left side of another of those, we will end with one of the following pair of relations:

\begin{eqnarray}
 \rho^{s_i}(a_i) = \rho^{s_{i+1}}(a_{i+1}) &and& \rho^{-s_i}(a_i) = \rho^{-s_{i+1}}(a_{i+1}) \label{eqn:typeIII-quasi-si}\\
 &or&\nonumber\\
 \rho^{r_i}(a_i) = \rho^{s_{i+1}}(a_{i+1}) &and& \rho^{-r_i}(a_i) = \rho^{-s_{i+1}}(a_{i+1}) \label{eqn:typeIII-quasi-ri} 
\end{eqnarray}

Lets work first with Eq. \eqref{eqn:typeIII-quasi-si}; after simple manipulations and a substitution, we get $\rho^{2(s_{i+1}-s_i)}(a_i) =a_i$. This relation, by lemma \ref{cor:rho-a=a}, implies that $s_{i+1} = s_i$. On Eq. \eqref{eqn:typeIII-quasi-si}, after canceling $\rho^{s_i} = \rho_{s_{i+1}}$ on both sides we obtain that $a_i=a_{i+1}$. As $s_i +r_i = \alpha = s_{i+1} + r_{i+1}$ we also got $r_i = r_{i+1}$.  So, summarizing,  Eq. \eqref{eqn:typeIII-quasi-si} implies that $s_i=s_{i+1}$, $r_i=r_{i+1}$ and $a_i=a_{i+1}$.

The same will happen on Eq. \eqref{eqn:typeIII-quasi-ri}, but we will get instead $s_i=r_{i+1}$, $r_i=s_{i+1}$ and $a_i=a_{i+1}$. 

Then the set of relations is equivalent to $a_1 = a_2 = \dots = a_n$ and the proportions of $r_i$'s and $s_i$'s saddles is the same on each edgepath. implying that the possible orientations of the candidate surface is parametrized by $a_1$, which have $2^{m - \max\{r_1, s_1\}}$ different values. So it has $m - \max\{r_1, s_1\}$ connected components. But recall that $r_1, s_1 < m/2$, then $m - \max\{r_1, s_1\} \geq \max\{r_1,s_1\} + 1$ implying that the candidate surface would be connected only when $r_1=s_1=0$ and $m=1$.
\end{proof}

Until this point, we have proved that for any Montesinos knot there are no ICON surfaces with disconnected boundary of type II or III. As a consequence, there will be no ICON surfaces with disconnected boundary on any alternating three tangle Montesinos knot because, as we will see, there are no type I candidate surfaces for alternating Montesinos knots.

\begin{theorem}\label{thm:prop-z-for-alt-mont}
 Alternating Montesinos knots of three tangles have property $\mathbb{Z}$.
\end{theorem}
\begin{proof}
 By  \cite{Alternating}, Montesinos knots reach their crossing number on a simplified Montesinos diagram of the form $M(p_1/q_1, \dots, p_n/q_n)$; this implies that if the Montesinos knot is alternating, the simplified diagram should be an alternating diagram. But a Montesinos diagram is alternating if and only if all rationals $p_i/q_i$ have the same sign.
 
 
 Without loss of generality, we can assume that $p_i/q_i >0$ for all $i$. Now, any edgepath $\gamma_i$ starts in a positive coordinate point and when it travels to the right it keeps with positive coordinate until $x=0$. So, there is no value $x>0$ such that $\gamma_1(x) + \gamma_2(x) + \dots + \gamma_n(x) = 0$, this is, there will be no candidate surfaces of type I. 
 
 By Theorems \ref{thm:no-typeii}  and \ref{thm:no-typeiii}, the type II or III ICON candidate surfaces have connected boundary. And by \cite{MR2221530} surfaces with connected boundary have property $\mathbb{Z}$.
 
 The only remaining possible candidate surfaces are the extended candidate surfaces. But, in the case of $n=3$, there are no extended candidate surfaces.
\end{proof}

We could prove theorem \ref{thm:prop-z-for-alt-mont} thanks to the non-existence of extended candidate surfaces, but for more than three tangles this surfaces exist; we were able to construct an extended candidate surface for a Montesinos knot with nine tangles.

\section{Pretzel knots}

As we saw in \cite{JFPretzel} there are ICON surfaces with disconnected boundary on the exterior of some pretzels $P(p,q,-r)$ and those knots have Property $\mathbb{Z}$. In this paper we will prove that for the other pretzels of the form $P(p,q,r)$ there are no ICON surfaces with disconnected boundary. So the result claimed in the abstract will be true; this is because for surfaces with connected boundary the result is well known to be true (see \cite{MR2221530}). 

In order to simplify the statement of the following theorem, we introduce the following definition:

\begin{definition}
 A triple $(a,b,c)$ of positive integers is an \emph{ICON-allowing triple} if it satisfies that:

   $$\frac{1}{a} + \frac{1}{b} = \frac{1}{c} \quad \text{and} \quad \frac{\text{gcd}(a,b)}{\text{gcd}(a,b,c)} \equiv  1 \pmod{2}$$
  
where $\text{gcd}$ is the greatest common divisor.

\end{definition}

\begin{theorem}[Main theorem]\label{thm:main-theorem}
 A pretzel knot $K= P(p,q,r)$ (with absolute values of $p$, $q$ and $r$ greater then one) has an ICON surface with disconnected boundary in its exterior if and only if $p$, $q$ and $r$ satisfy that: $$(p,q,r) = \pm (2a+1,2b+1, -2c-1)$$ for some ICON allowing triple $(a,b,c)$.
 
\end{theorem}

The existence of ICON surfaces and the property $\mathbb{Z}$ are valid under permutations of the braids of the Pretzel and also under a reflexion (changing the sign of the three braids). For that reason, the possibilities for $p,q$ and $r$ can be reduced to the following four:
\begin{enumerate}
 \item $p$, $q$ and $r$ positive integers. 
 \item $p$, $q$ odd positive and $r$ odd negative (studied in \cite{JFPretzel}).
 \item $p$, $q$ odd positive, $r$ even negative.
 \item $p$ odd positive,  $q$ odd negative and $r$ even positive.
\end{enumerate}

Now we divide the proof of the main theorem \ref{thm:main-theorem} on three lemmas (\ref{lemma:popopo}, \ref{lemma:popone} and \ref{lemma:ponope}), each lemma covers a different case from the ones described above. 

\begin{lemma}\label{lemma:popopo}
 There are no ICON surfaces with disconnected boundary in the knot exterior of a pretzel knot of the form $P(p,q,r)$ with $p$, $q$ and $r$ positive odd integers greater than 1.
\end{lemma}

\begin{proof}
 As we can easily see, there are no Type I candidate surfaces, because no edgepath system will satisfy condition (E4).  Type II and III candidate surfaces are covered by Theorems \ref{thm:no-typeii} and \ref{thm:no-typeiii}.
\end{proof}

\begin{lemma}\label{lemma:popone}
  There are no ICON surfaces with disconnected boundary in the knot exterior of a pretzel knot of the form $P(p,q,-r)$ with $p, q$ and $r$ integers greater than 1; $p$, $q$ odd and $r$ even.
\end{lemma}
\begin{proof}
 In this case, there are candidate surfaces of type I. But we can ignore them because none of them has boundary slope equal to zero. This can be proved as follows. 
 
 Let $\gamma_p$, $\gamma_q$ and $\gamma_r$ be the edgepath system corresponding to a type I candidate surface. As we can easily check, each edgepath has only two directions: always decreasing or always increasing. So the twist number can be bounded by the length of these two directions; this is,
 \begin{eqnarray*}
  2> &\tau(\gamma_p)& >-2(p-1) \\
  2> &\tau(\gamma_q)& >-2(q-1) \\
  2(r-1)> &\tau(\gamma_r)& >-2 
 \end{eqnarray*}

As we know, the twist of the Seifert surface is $\tau(F_0)=2-2(p + q)$, and so we get that 
$$\text{slope}(F) = \tau(F) - \tau(F_0) = \tau(\gamma_p)  + \tau(\gamma_q) + \tau(\gamma_r) + 2(p+q) - 2 > 0$$
Then, the slope of $F$ is never zero.

By Theorems \ref{thm:no-typeii} and \ref{thm:no-typeiii}, this completes the proof.
\end{proof}

In the following final case, we have to work harder to eliminate the possibility of a type I candidate ICON surface. 

\begin{lemma}\label{lemma:ponope}
  There are no ICON surfaces with disconnected boundary in the knot exterior of a pretzel knot of the form $P(p,-q,r)$ with $p, q$ and $r$ integers greater than 1; $p$, $q$ odd and $r$ a even. 
\end{lemma}

\begin{proof}

By Theorems \ref{thm:no-typeii} and \ref{thm:no-typeiii}. we only need to prove that there are no type I candidate surface. To do that, we start with the cases when no edgepath is constant. 

\textbf{Case 1: No edgepath is constant.}
In this case, each edgepath has two possibilities: to be a decreasing function or an increasing function. On the following table we write the equation for each possibility:

\begin{center}
  \begin{tabular}{c|l|l}
    edgepath & decreasing function & increasing function \\\hline
    $\gamma_p(x)$ & $1-x$ & $\frac{x}{p-1}$\\
    $\gamma_{-q}(x)$ & $\frac{-x}{q-1}$ & $x-1$ \\
    $\gamma_r(x)$ & $1-x$ & $\frac{x}{r-1}$ \\    
  \end{tabular}
\end{center}

Let start by considering the case in which one of the edgepaths $\gamma_p$ or $\gamma_r$ is decreasing (value equal to $1-x$). Without lost of generality assume that $\gamma_p(x) = 1-x$, then observe that $\gamma_{-q}(x) \geq x-1$, this implies that the sum $\gamma_p(x) + \gamma_{-q}(x) + \gamma_r(x) \geq \gamma_r(x) > 0$. So there will be no candidate surface by condition E1.

Then, both $\gamma_p$ and $\gamma_r$ have to be increasing functions, meaning that $\gamma_p(x) = x/(p-1)$ and $\gamma_r(x) = x/(r-1)$.  Now, we analyze the two possibilities for $\gamma_{-q}$ .

\textbf{Case 1.1 $\gamma_{-q}(x) = -x/(q-1)$}

In this case, the only way that equation $\gamma_p(x) + \gamma_{-q}(x) + \gamma_r(x)=0$ has a solution is when $p,q$ and $r$ satisfy:
\begin{equation}\label{eqn:condition-case+-+}
\frac{1}{p-1} +\frac{1}{r-1} = \frac{1}{q-1}  
\end{equation}

Observe that the three edgepaths have length at most 1; moreover we have the following inequalities:
\begin{eqnarray*}
 0 < \tau(\gamma_p) < 2 \\
 -2 < \tau(\gamma_{-q}) < 0 \\
 0 < \tau(\gamma_r) < 2
\end{eqnarray*}

This implies that $ -2 <  \tau(\gamma_p) + \tau(\gamma_{-q}) +\tau(\gamma_r) < 4$. But we want this twist to be equal to the twist of the Seifert surface (in order to obtain a  zero boundary slope surface); then $ -2< 2(q-p)+2 < 4$ which is equivalent to $-2 < q -p < 1$. But as $q$ and $p$ are odd their difference is even, so $p = q$. But this, together with condition (\ref{eqn:condition-case+-+}), implies that $1/(r-1) = 0$, which is impossible.

\textbf{Case 1.2: $\gamma_{-q}(x) =x -1$}

In this case, the equation $\gamma_p(x) + \gamma_{-q}(x) + \gamma_r(x)=0$ gets the form, 
$$\frac{x}{p-1}+x-1+\frac{x}{r-1} = 0$$ whose solution is:
$$x=\frac{(p-1)(r-1)}{(p-1)(r-1) + (p-1)+(r-1)}$$

Then we compute the twist $\tau$ using formula \eqref{eqn:tao_gamma_i}:
\begin{multline}
\tau(F) = \tau(\gamma_p) + \tau(\gamma_{-q}) + \tau(\gamma_r) = \tau(\gamma_p^-) - \tau(\gamma_q^+) + \tau(\gamma_r^-) \\
\phantom{\tau(F) } = \frac{2(p-1)}{p+r-2} + 2q - 2\frac{(p-1)(r-1)+(p-1)+(r-1)}{(p-1)+(r-1)} + \frac{2(r-1)}{p+r-2} \\
\phantom{\tau(F) } = 2+2q -2\frac{(p-1)(r-1)+(p-1)+(r-1)}{(p-1)+(r-1)}
\end{multline}

Now, the twist of a Seifert surface is $\tau(F_0) = 2(q-p) + 2$ then, $F$ will have zero slope only if $$2(q-p) + 2 = 2+2q -2\frac{(p-1)(r-1)+(p-1)+(r-1)}{(p-1)+(r-1)}$$

After some simple manipulations we obtain that $p=1$, but this is impossible since we require $p>1$. This completes Case 1.2 and by consequence Case 1.

\textbf{Case 2: At least one constant edgepath.}
If one edgepath is constant, then we are looking for a solution of \eqref{eqn:condit-E3-eqn} but with $x \geq \min\{1-1/p, 1-1/q, 1-1/r\} = 1 - 1/\min\{p,q,r\}$. 

First, lets prove that $\min\{p,q,r\}=q$.  If $\min\{p,q,r\}<q$, then we will have that $p=\min\{p,q,r\}$ or $r=\min\{p,q,r\}$; without loss of generality, assume that $p=\min\{p,q,r\}<q$. 

Under these conditions, if there is a solution of \eqref{eqn:condit-E3-eqn} with one constant edgepath, we necessary will get that $\gamma_p \equiv \langle 1/p \rangle$ (is constant). Then we are looking for a solution of \eqref{eqn:condit-E3-eqn} where $1 > x \geq (p-1)/p$. 

Notice that $\gamma_r(x) > 0$ and $\gamma_{-q}(x) > \min\{x-1,-1/q\}$,  then 
$$\gamma_p(x) + \gamma_{-q}(x)+ \gamma_r(x) > \frac{1}{p} + \min\{x-1,1/q\} = \min\{x-\frac{p-1}{p}, 1/p-1/q\} \geq 0$$

This implies that there will be no solution in this case. The same occurs when $r=\min\{p,q,r\}$. So, we must have that $q = \min\{p,q,r\}$. As we are looking for solutions with 
$x \geq \min\{1-1/p, 1-1/q, 1-1/r\} = 1-1/q$ then $\gamma_q \equiv \langle -1/q \rangle$ (is constant).

Now, notice that it is impossible for the three edgepaths to be constant, because of the parity conditions on $p$, $q$ and $r$, they can not satisfy condition E3. So, the only possibilities are when just one or two of the edgepaths are constant.

\textbf{Case 2.1: Only one edgepath is constant}

By the observations made previously, the constant edgepath is $\gamma_{-q}$. The other two are non-constant, so we now analyze the possibilities, and solve the corresponding equations and compute the slope. 

We begin with the case where $\gamma_p = \gamma_p^-$ and $\gamma_r = \gamma_r^-$. Observe that the twist for each of this edgepath is bounded between 0 and 2 (never reaches 0, neither 2). So, the twist $\tau(F)$ of any candidate surface $F$ associated to this edgepath satisfies $0<\tau(F)< 4$. But, the twist of Seifert is $2q -2p + 2$ which is an even integer, and the only even integer between $0$ and $4$ (not including 0, neither 4) is $2$, so $2q-2p + 2 = 2$ then $p=q$. As $\gamma_{-q}$ is constant and $p=q$ then $\gamma_{p}$ has to be constant, implying that $2 = \tau(F) = \tau(\gamma_r)$  it is impossible for $\tau(\gamma_r)$ to be 2 since this implies $\gamma_{-r}$ reaches $x=0$.

For the other three cases we will compute the possible values that satisfy \eqref{eqn:condit-E3-eqn}. The following table summarizes those computations. 

\begin{center}
  \begin{tabular}{c|p{4cm}|l}
    $\gamma_p$, $\gamma_r$ & Equation from condition E3 & Solution  \\\hline
    $++$ & $(1-x) + (1-x) = \frac{1}{q} $ \vspace{1.5ex}& $x = \frac{2q-1}{2q}$   \\
    $+-$ & $1-x+ \frac{x}{r-1} = \frac{1}{q} $ \vspace{1.5ex}& $x = \frac{q-1}{q}\frac{r-1}{r-2}$   \\
    $-+$ & $\frac{x}{p-1} + 1-x = \frac{1}{q} $\vspace{1.5ex}& $x = \frac{q-1}{q}\frac{p-1}{p-2}$  \\
  \end{tabular}
\end{center}

We can discard the $++$ case because the solution corresponds to vertexes at end of $\gamma_p$ and $\gamma_r$; so by Corollary \ref{cor:vertexes-implies-coonected-bounday} the associated candidate surface will have the number of sheets equal to one, so the boundary will be connected. 

For the other two cases, we now compute the $x$ coordinate when the twist is equal to Seifert twist. Because we want zero slope candidate surfaces, this $x$ value has to be equal to the computed above.

\begin{center}
  \begin{tabular}{c|p{5cm}|l}
    $\gamma_p$, $\gamma_r$ & Equation Twist = Seifert Twist & Solution  \\\hline
    $+-$ & $\frac{2}{1-x} - 2p + 2 - \frac{2x}{(1-x)(r-1)}$ $= 2q~-~2p~+~2$ \vspace{1.5ex}& $x = \frac{(q-1)(r-1)}{q(r-1)-1}$ \\
    $-+$ & $2 - \frac{x}{1-x}\frac{2}{p-1} + \frac{2}{1-x}- 2r$ $= 2q~-~2p~+~2$ & $x = \frac{(p-1)(q+r-p-1)}{(p-1)(q+r-p) -1}$ \\
  \end{tabular}
\end{center}

In the case $+-$, the equation \eqref{eqn:case2.1_+-case} is the condition to get a candidate surface of zero slope. After applying some basic algebraic manipulations the condition turns out to be equivalent to $q=1$, which contradicts the hypothesis about $q$.

\begin{equation}\label{eqn:case2.1_+-case}
\frac{q-1}{q}\frac{r-1}{r-2} =  \frac{(q-1)(r-1)}{q(r-1)-1}
\end{equation}

In the case $-+$, we write again an equation, but the equation can not be reduced as above, so we only apply some basic algebraic manipulations and write it down in the form \eqref{eq:case2.1_-+case}
\begin{equation}\label{eq:case2.1_-+case}
\frac{(p-2)(q+r-p-1)}{(p-1)(q+r-p) -1} = \frac{q-1}{q}
\end{equation}

Reducing modulo 2 (taking to account that $p,q \equiv 1$ and $r \equiv 0$ modulo 2) we obtain a contradiction.

\textbf{Case 2.2: Exactly two edgepaths are constant}

When two edgepaths are constant the twist of the candidate surface is equal to the twist of the non-constant edgepath. As we want the slope of the surface to be zero, then the twist has to be an even integer, meaning that the non-constant edgepath ends at a vertex. Then, by Corollary \ref{cor:vertexes-implies-coonected-bounday}, the candidate surface will have connected boundary.

\end{proof}

\appendix 
\section*{Final comments}
The results obtained here are a step toward proving that all alternating Montesinos knots have property $\mathbb{Z}$, the remaining step is to solve the case for extended candidate surfaces. The prove for three tangles pretzels give us an idea of how difficult it gets when we try to prove property $\mathbb{Z}$ for non-alternating knots.  This tell us that for the general case of Montesinos knots (not necessarily alternating) we have to improve our techniques.

\section*{Acknowledgments}
I am grateful to Victor N\'u\~nez and Enrique Ram\'{\i}rez for stimulating conversations. This Work was sponsored by CIMAT and CONACyT.

\bibliography{extras/bibliografia}

\begin{thebibliography}{6}
\providecommand{\natexlab}[1]{#1}
\providecommand{\url}[1]{\texttt{#1}}
\expandafter\ifx\csname urlstyle\endcsname\relax
  \providecommand{\doi}[1]{doi: #1}\else
  \providecommand{\doi}{doi: \begingroup \urlstyle{rm}\Url}\fi

\bibitem[Dunfield(2001)]{NathanM2001309}
Nathan~M. Dunfield.
\newblock {A table of boundary slopes of Montesinos knots}.
\newblock \emph{Topology}, 40\penalty0 (2):\penalty0 309--315, 2001.
\newblock ISSN 0040-9383.
\newblock \doi{10.1016/S0040-9383(99)00064-6}.
\newblock URL
  \url{http://www.sciencedirect.com/science/article/pii/S0040938399000646}.

\bibitem[Eudave-Mu{\~n}oz(2013)]{eudave-munoz2013}
Mario Eudave-Mu{\~n}oz.
\newblock {On knots with icon surfaces}.
\newblock \emph{Osaka Journal of Mathematics}, 50\penalty0 (1):\penalty0
  271--285, 03 2013.
\newblock URL \url{http://projecteuclid.org/euclid.ojm/1364390429}.

\bibitem[Gonz{\'a}lez-Acu{\~n}a and Ram{\'i}rez(2006)]{MR2221530}
Francisco Gonz{\'a}lez-Acu{\~n}a and Arturo Ram{\'i}rez.
\newblock {A knot-theoretic equivalent of the {K}ervaire conjecture}.
\newblock \emph{J. Knot Theory Ramifications}, 15\penalty0 (4):\penalty0
  471--478, 2006.
\newblock ISSN 0218-2165.

\bibitem[Hatcher and Oertel(1989)]{MR1030987}
A.~Hatcher and U.~Oertel.
\newblock {Boundary slopes for {M}ontesinos knots}.
\newblock \emph{Topology}, 28\penalty0 (4):\penalty0 453--480, 1989.
\newblock ISSN 0040-9383.

\bibitem[Lickorish and Thistlethwaite(1988)]{Alternating}
W.~B.~R. Lickorish and M.~B. Thistlethwaite.
\newblock {Some links with nontrivial polynomials and their crossing-numbers}.
\newblock \emph{Comment. Math. Helv.}, 63\penalty0 (4):\penalty0 527--539,
  1988.
\newblock ISSN 0010-2571.
\newblock \doi{10.1007/BF02566777}.
\newblock URL \url{http://dx.doi.org/10.1007/BF02566777}.

\bibitem[Rodr{\'i}guez-Viorato and {Gonzal{\'e}z Acu{\~n}a}(2016)]{JFPretzel}
Jes{\'u}s Rodr{\'i}guez-Viorato and Francisco {Gonzal{\'e}z Acu{\~n}a}.
\newblock {On pretzel knots and Conjecture $\mathbb{Z}$}.
\newblock \emph{Journal of Knot Theory and Its Ramifications}, 25\penalty0
  (02):\penalty0 1650012, 2016.
\newblock \doi{10.1142/S0218216516500127}.
\newblock URL
  \url{http://www.worldscientific.com/doi/abs/10.1142/S0218216516500127}.

\end{thebibliography}

\end{document}